\documentclass[reqno, 12pt]{amsart}
\usepackage{amsmath,amsthm,amsfonts,amssymb}

\date{}
\newtheorem{theorem}{Theorem}[section]
\newtheorem{lemma}[theorem]{Lemma}
\newtheorem{corollary}[theorem]{Corollary}
\newtheorem{remark}[theorem]{Remark}
\newtheorem{proposition}[theorem]{Proposition}
\newtheorem{example}[theorem]{Example}
\numberwithin{equation}{section}

\begin{document}

\centerline{\large{\bf On the uniqueness of solutions to continuity equations}}

\vspace*{0.2cm}

\centerline{\bf V.I. Bogachev$^{a}$\footnote{corresponding author.
{\it E-mail addresses}: vibogach@mail.ru (V. Bogachev), daprato@sns.it (G. Da Prato),
roeckner@math.uni-bielefeld.de (M. R\"ockner), starticle@mail.ru (S. Shaposhnikov)},
G. Da Prato$^{b}$,  M. R\"ockner$^{c}$,
S.V. Shaposhnikov$^{a}$}

\vspace*{0.2cm}

\small{
$^{a}$
{\it Department of Mechanics and Mathematics, Moscow State
University, 119991 Moscow, Russia and National Research University Higher School of Economics, Moscow, Russia}

$^{b}$  {\it Scuola Normale Superiore di Pisa, Piazza dei Cavalieri, 56100 Pisa, Italy}

$^{c}$
{\it Fakult\"at f\"ur Mathematik, Universit\"at Bielefeld,
D-33501 Bielefeld, Germany}
}

\vspace*{0.2cm}

\noindent
{\bf Abstract}

We obtain sufficient conditions for
the uniqueness of solutions to the Cauchy problem for the continuity equation
in classes of measures that need not be absolutely continuous.

 \vspace*{0.2cm}


\section*{Introduction}

In this paper we study the uniqueness problem for the continuity equation
$$
\partial_t\mu_t+{\rm div}(b\mu_t)=0
$$
with respect to measures on $\mathbb{R}^d$.
We consider solutions given by families of locally bounded Borel measures $(\mu_t)_{t\in[0, T]}$.
A precise definition is given below.

There is a vast literature devoted to uniqueness and existence problems for the Cauchy problem for
such equations. An important problem is to specify a class of measures  $\mu_t$ in which, under reasonable
assumptions about coefficients and initial data, there is a unique solution to the Cauchy problem.
Certainly, if the coefficients are sufficiently regular, say, Lipschitzian or satisfy the Osgod type condition,
then we can take the whole class
of bounded measures $\mu_t$ (see, e.g., \cite{AB}).
According to a well-known result of Ambrosio \cite{A08} (see also \cite{Man} for equations with a potential term)
on representations of nonnegative bounded solutions by means of averaging
with respect to measures concentrated
on solutions to the corresponding  ordinary equation $\dot{x}=b(x, t)$,
any uniqueness condition for the  ordinary equation guarantees
uniqueness in the class of nonnegative bounded measures.
However, in the class of signed measures there is no such representation.

In the case of non smooth coefficients a class convenient in many respects  is the class of
measures absolutely continuous with respect to Lebesgue measure, which
is quite natural, in particular, taking into account existence results.
A~study of this class initiated by Cruzeiro \cite{Cruz1}, \cite{Cruz2}
and DiPerna and Lions \cite{DiPL} was continued by many researchers.
A~large number of papers are devoted to the so-called Lagrangian flows and their generalizations
(see \cite{A04} and~\cite{A08}).
However, this class of absolutely continuous measures
is rather narrow, in particular, it does not enable one
to deal with singular initial data (and is essentially oriented towards vector fields having at least some
minimal regularity such as the existence of their divergence or being BV).
In addition, this class has no universal analogs in infinite dimensions.
There are several papers (see \cite{Cruz1}, \cite{P}, \cite{BMW}, \cite{UZ}, \cite{A04}, \cite{DPFR},
and~\cite{KR}) concerned with the infinite-dimensional case and using as reference measures
certain special measures (all reducing to the absolute continuity with respect  to Lebesgue measure
when the infinite-dimensional state space is replaced by~$\mathbb{R}^d$) such as Gaussian measures,
convex measures, and differentiable measures, which becomes
rather restrictive in infinite dimensions in  spite of importance of such classes of measures
in applications (e.g., when they are Gibbs measures) .
The recent paper \cite{AT} develops continuity equations in metric measure spaces, but again only considering solutions
absolutely continuous with respect to the underlying fixed measure.

Thus, it is natural to look for other classes of measures,
apart from absolutely continuous measures, in which the existence and
uniqueness of solutions hold in the case of non-smooth coefficients.
In this paper we consider the finite-dimensional case; it turns out that even in the one-dimensional
case in the present framework new results can be obtained.

The main result in this direction
obtained  in our paper can be briefly formulated as follows:
uniqueness holds in a certain class of measures with respect to which the given vector field  $b$
can be suitably approximated by smooth vector fields (thus, the uniqueness class may depend on~$b$).
A~precise formulation is given below  (see Theorems~\ref{th1} and~\ref{th2}), but we observe that this
result is consistent with typical methods of constructing
solutions when $b$ is approximated by smooth fields $b_k$ and the solution is obtained as a limit point
of the sequence of solutions $\mu_t^k$ for~$b_k$.
As in many existing papers, our conditions admit discontinuous fields,
but the hypotheses are mostly incomparable
(with the already cited papers and, e.g., \cite{BC}, \cite{BJ}, \cite{CrD}, \cite{Del}, \cite{Del08},
\cite{Des}, and~\cite{HLB}).

As an application we obtain
some new results for the continuity equation with a merely continuous coefficient $b$.
In particular, we substantially improve the recent result from \cite{Cr},
where the uniqueness is proved in dimension one
for absolutely continuous solutions under the assumptions
that $b$ is continuous and nonnegative,
the trajectories of $\dot{x}=b(x)$ do not blow up in finite time and the set
of zeros $Z=\{x\colon b(x)=0\}$ consists of a finite union of points
and closed intervals.
For example, we prove the following assertion: if $b\in C(\mathbb{R})$
and $0\le b(x)\le C+C|x|$ (which is a constructive condition to exclude a blowup), then
the uniqueness holds in the class of all
locally bounded {\rm(}possibly signed{\rm)} measures $\mu$ on $\mathbb{R}^d\times[0, T]$
given by families of locally bounded measures $(\mu_t)_{t\in[0, T]}$  such that
$$
|\mu_t|(\partial Z)=0 \quad \hbox{\rm for almost all} \quad t\in[0, T],
$$
where $\partial Z$ is the
boundary of the set $Z=\{x\colon\ b(x)=0\}$. So there are no structure restrictions
on the zero set of $b$ and no assumption that $\mu$ has a density.
Moreover, we consider the multidimensional case.
Finally, we prove an existence result which produces solutions with our uniqueness
properties.

\section{Main result}

Let us consider the Cauchy problem
\begin{equation}\label{e1}
\partial_t\mu+{\rm div}(b\mu)=0, \quad \mu|_{t=0}=\nu,
\end{equation}
where $\nu$ is a locally bounded Borel measure on $\mathbb{R}^d$, i.e., a real function on the
class of all bounded Borel sets in $\mathbb{R}^d$ that is countably additive on
the class of Borel subsets of every compact set. Locally bounded Borel measures
on $\mathbb{R}^d\times [0,T]$
are defined similarly.
In particular,   for every compact set $K\subset \mathbb{R}^d\times [0,T]$,
the total variation of a locally
bounded Borel measure $\mu$ on $K$ (denoted by $|\mu|(K)$) is finite.
A Borel measure is  bounded if it has a finite total variation.

We say that a locally bounded Borel measure $\mu$ on $\mathbb{R}^d\times[0, T]$
is given by a family of Borel locally bounded measures $(\mu_t)_{t\in [0, T]}$ on $\mathbb{R}^d$
if, for every bounded Borel set $B\subset \mathbb{R}^d$, the mapping $t\mapsto \mu_t(B)$ is measurable,
$|\mu_t|(K)\in L^1[0, T]$
for every compact $K\subset \mathbb{R}^d$ and
$$
\int_0^T\int_{\mathbb{R}^d}u(x, t)\,\mu(dxdt)=\int_0^T\int_{\mathbb{R}^d}u(x, t)\,\mu_t(dx)\,dt
\quad \forall\, u\in C^{\infty}_0(\mathbb{R}^d\times(0, T)).
$$
Clearly, the previous equality extends to bounded Borel measurable functions $u$
that vanish for $x$ outside a ball.
In the considered framework there is no difference between measures on
$\mathbb{R}^d\times [0,T]$ and on $\mathbb{R}^d\times (0,T)$, because we study only measures
of the above form represented by families of measures on~$\mathbb{R}^d$ via Lebesgue measure.
Not every measure on  $\mathbb{R}^d\times [0,T]$ has this property, of course.

A Borel locally bounded measure $\mu$ on $\mathbb{R}^d\times[0, T]$ given by
a family of locally bounded measures $(\mu_t)_{t\in[0, T]}$
is a solution to the Cauchy problem (\ref{e1}) if $b\in L^1(|\mu|, U\times [0, T])$
for every ball $U\subset\mathbb{R}^d$ and, for every  function
$u$ in the class $C^{1, 1}(\mathbb{R}^d\times [0, T))$  (consisting of functions that are continuous
on $\mathbb{R}^d\times [0, T)$ along with their first order derivatives
in~$t$ and~$x$)
 such that $u(x, t)\equiv 0$ if $|x|>R$ for some~$R$, one has
\begin{equation}\label{ee1}
\int_{\mathbb{R}^d}u(x, t)\,\mu_t(dx)=\int_{\mathbb{R}^d}u(x, 0)\,\nu(dx)
+\int_0^t\int_{\mathbb{R}^d}\Bigl[\partial_tu+\langle b, \nabla u\rangle\Bigr]\,\mu_s(dx)\,ds
\end{equation}
for almost every $t\in[0, T]$. Note that the integrals exist for almost all $t$.

Let us observe that it is possible to pick a common full measure set $S\subset [0,T]$ such that
(\ref{ee1}) will hold for each $t\in S$ for all admissible functions~$u$. Indeed, such a set
exists for a countable collection of functions~$u$, so it is sufficient to choose this countable
 collection  in such a way that for any function $u$ in the indicated class there is a sequence
 $\{u_n\}$ in this collection that is uniformly bounded along with
 the derivatives in $t$ and $x$ and $u_n\to u$, $\partial_t u_n\to \partial_t u$,
 $\nabla u_n\to \nabla u$ pointwise.

\begin{theorem}\label{th1}
Let $\mu^1=\mu^1(dx)\,dt$ and $\mu^2=\mu^2(dx)\,dt$ be two solutions to the
Cauchy problem~{\rm (\ref{e1})}.
Assume that the measure $\mu=\mu_t(dx)\,dt$, where $\mu_t=\mu_t^1-\mu_t^2$, satisfies the following conditions:
for some number $C_1\ge 0$ and
for every ball $U\subset\mathbb{R}^d$ one can find a number $C_2\ge 0$,
a sequence of vector fields $b_k\in C^{\infty}(\mathbb{R}^d\times \mathbb{R}^1,\mathbb{R}^d)$ and
a sequence of positive functions $V_k\in C^1(\mathbb{R}^d)$
{\rm(}all depending on the considered measure $\mu$, the constant $C_1$ and the ball $U${\rm)} such that

\, {\rm (i)} \, $|b_k(x, t)|\le C_1+C_1|x|$,

\, {\rm (ii)} \,
$\inf_k\inf_{U}V_k(x)>0$ and
$$
\langle b_k(x, t), \nabla V_k(x)\rangle \le
\Bigl(C_2-2\max_{|\xi|=1}\langle\mathcal{B}_k(x, t)\xi, \xi\rangle\Bigr)V_k(x)
$$
for every $k$ and $(x, t)\in U\times[0, T]$, where $\mathcal{B}_k=(\partial_{x_i}b^j_k)_{i,j\le d}$,

\, {\rm (iii)} \, $\lim\limits_{k\to\infty}\|(b_k-b)\sqrt{V_k}\|_{L^1(|\mu_t|\,dt, \,U\times[0, T])}=0$.

Then $\mu=0$, i.e., $\mu^1=\mu^2$.
\end{theorem}

Note that the equality $\mu^1=\mu^2$ is equivalent to the equality $\mu^1_t=\mu^2_t$
for almost every $t\in [0,T]$; changing $\mu_t$ for $t$ in a measure zero set we do not
change the solution. However, under broad assumptions about $b$, one can find
a version of $t\mapsto \mu_t$ that is continuous in the sense of generalized functions
(or even weakly continuous in the case of probability measures), and then such versions
are uniquely defined. Moreover, in that case (\ref{ee1}) holds poitwise and,
as $t\to 0$, the measures $\mu_t$ converge to the initial
measure $\nu$ in the respective sense (as distributions or weakly). This can be done
by taking  in (\ref{ee1}) functions $u$ independent of~$t$ (see \cite[Lemma~2.1]{BDPR}).

So in order to find a uniqueness class  for the Cauchy problem (\ref{e1}) one should find
approximations of $b$ satisfying conditions (i), (ii) and describe measures satisfying condition~(iii).

Let us consider the standard example of an ordinary differential equation on the real line
without uniqueness and see what kind of uniqueness for the associated continuity equation
is offered by our theorem.

\begin{example}\rm
Let $d=1$ and $b(x)=\sqrt{|x|}$.
Set $b_k(x)=(x^2+k^{-2})^{1/4}$.
It is obvious that
$$|b(x)-b_k(x)|\le k^{-1/2}.$$
Let us calculate $b_k'(x)$:
$$
b_k'(x)=\frac{x}{2(x^2+k^{-2})^{3/4}}.
$$
Finally, we take
$$V_k(x)=\frac{1}{b_k(x)^2}=\frac{1}{(x^2+k^{-2})^{1/2}}.$$
Clearly, (ii) is fulfilled.
Moreover, we have
$$
g_k(x):=|b(x)-b_k(x)|\sqrt{V_k(x)}=|b(x)-b_k(x)|b_k(x)^{-1}.
$$
Note that $g_k(x)=1$ if $b(x)=0$, i.e., $x=0$. If $x\neq 0$, then
$$
g_k(x)\le\frac{k^{-1/2}}{|x|^{1/2}}\to 0 \quad \hbox{if} \quad k\to\infty.
$$
Thus $|g_k|\le 1$ and $g_k\to I_{\{0\}}$, where $I_{\{0\}}$ is the indicator
of the set $\{b=0\}=\{0\}$. Let $\mu=\mu_t\,dt$ be a locally bounded
measure on $\mathbb{R}^1\times[0, T]$.
According to the Lebesgue dominated convergence theorem, one has
$$
\lim_{k\to \infty}\int_0^T\int_{U}|b-b_k|\sqrt{V_k}\,d|\mu_t|\,dt=
\int_0^T|\mu_t|(\{x\in U\colon\ b(x)=0\})\,dt
$$
for every interval $U$.
Hence the uniqueness holds in the class of all locally bounded (possibly, signed)
solutions $\mu=\mu_t\,dt$ such that $|\mu_t|(\{0\})=0$ for almost all $t\in[0, T]$.

It is interesting that this  result is sharp:
there exist two solutions $\mu_t^1\equiv \delta_0$ and $\mu_t^2=\delta_{t^2/4}$
to the Cauchy problem with $b(x)=\sqrt{|x|}$ and $\nu=\delta_0$.
Note that only $\mu_t^2$ belongs to our  uniqueness class and that this solution is not absolutely
continuous.
\end{example}

The proof of Theorem \ref{th1} is based on the maximum principle and Holmgren's principle.

\begin{lemma}\label{l01}
Suppose that $h\in C^{\infty}(\mathbb{R}^d\times \mathbb{R}^1,\mathbb{R}^d)$,
$|h(x, t)|\le C_1+C_1|x|$
for some number $C_1>0$ and for all $(x, t)\in \mathbb{R}^d\times[0, T]$.
Assume also that for some positive function $V\in C^1(\mathbb{R}^d)$,
and number $C_2>0$ one has
$$
\langle h(x, t), \nabla V(x)\rangle\le
\Bigl(C_2-2\sup_{|\xi|=1}\langle\mathcal{H}(x, t)\xi, \xi\rangle\Bigr)V(x),
$$
for all $(x, t)\in \mathbb{R}^d\times[-1, T]$, where
$\mathcal{H}(x, t)=(\partial_{x_i}h^j(x, t))_{1\le i, j\le d}$.

Then, for any $s\in (0,T)$, the Cauchy problem
$$
\partial_tf+\langle h, \nabla f\rangle=0, \quad f|_{t=s}=\psi,
$$
where $\psi\in C^{\infty}_0(\mathbb{R}^d)$,
has a smooth solution $f$ on $\mathbb{R}^d\times (-1,s]$ such that
$$
|f(x, t)|\le \max_y |\psi(y)|, \quad
|\nabla f(x, t)|^2\le V(x)e^{C_2(s-t)}\max_y \frac{|\nabla\psi(y)|^2}{V(y)}, \quad t\in [0,s].
$$
Moreover, $f(x, t)=0$ if $|x|>R$ for some number $R=R(s, \psi, C_1)>0$.
\end{lemma}
\begin{proof}
The solution $f$ is given by the equality $f(x_0, t_0)=\psi(x(s))$,
where $x(\cdot)$ is a solution to the ordinary equation $\dot{x}(t)=h(x(t), t)$, $x(t_0)=x_0$.
Note that
$$\frac{d}{dt}|x(t)|^2=2\langle h(x(t), t), x(t)\rangle\ge -C_1-3C_1|x(t)|^2.$$
Hence
$$
|x(t)|^2\ge |x_0|^2e^{3C_1(t_0-t)}+C_1(t_0-t)e^{-3C_1t}\ge |x_0|^2e^{-3C_1s}-C_1s.
$$
Assume that $\psi(x)=0$ if $|x|>r$. Then there exists a number $R=R(s, r, C_1)$ such that
$$
|x_0|^2e^{-3C_1s}-C_1s\ge r \quad\hbox{if} \quad |x_0|\ge R.
$$
 Therefore,  $|x(s)|\ge r$ and $f(x_0, t_0)=0$.

Let us prove the announced gradient estimate.
Set  $u=2^{-1}\sum_{k=1}^{d}|\partial_{x_{k}}f|^{2}$.
Differentiating the equation $\partial_{t}f+\langle h, \nabla f\rangle=0$
with respect to $x_{k}$ and multiplying by $\partial_{x_{k}}f$ we find that
$$
\partial_{t}u+\langle h, \nabla u\rangle+\langle\mathcal{H}\nabla f,\nabla f\rangle=0.
$$
Since
$\langle\mathcal{H}\nabla f,\nabla f\rangle\le 2u\sup_{|\xi|=1}\langle\mathcal{H}(x, t)\xi, \xi\rangle,$
we have
$$
\partial_tu+\langle h, \nabla u\rangle+
2u\sup_{|\xi|=1}\langle\mathcal{H}\xi, \xi\rangle\ge 0.
$$
Set  $u=wV$. Then
$$
\partial_{t}w+\langle h, \nabla w\rangle+qw\ge0,
$$
where
$$
q=2\sup_{|\xi|=1}\langle\mathcal{H}\xi, \xi\rangle
+\langle h, \nabla V\rangle V^{-1}\le C_2.
$$
Note that $w(x, s)\le\max_y |\nabla\psi(y)|^2/V(y)$.
Then the maximum principle (see \cite[Theorem 3.1.1]{StrV}) yields that
$w(x, t)\le e^{C_2(s-t)}\max_y |\nabla\psi(y)|^2/V(y)$, which completes the proof.
\end{proof}

\begin{proof}[Proof of Theorem \ref{th1}]
Let us fix $\psi\in C^{\infty}_0(\mathbb{R}^d)$ and $s$ in the full measure set $S\subset (0, T)$
mentioned before the theorem (such that (\ref{ee1}) holds for all admissible $u$ and each $t\in S$).
Let $R=R(s, \psi, C_1)$ be a number from Lemma \ref{l01}. According to the hypotheses of the theorem,
for the ball $\{x\colon |x|<2R\}$ there exist sequences $b_k$ and $V_k$
satisfying all conditions (i)--(iii).

According to Lemma \ref{l01} applied with $h=b_k$ there exists a smooth
solution $f$ to the Cauchy problem
$$
\partial_tf+\langle b_k, \nabla f\rangle=0, \quad f|_{t=s}=\psi
$$
satisfying the estimate
$$
|f(x, t)|\le\max_y |\psi(y)|, \quad |\nabla f(x, t)|^2
\le V_k(x)e^{C_2(s-t)}\max_y \frac{|\nabla\psi(y)|^2}{2V_k(y)}, \quad (x,t)\in \mathbb{R}^d\times [0,s].
$$
Moreover, $f(x, t)=0$ if $|x|>R$.
Certainly, $f$ depends on several parameters ($k$, $s$, $\psi$, etc.), which is suppressed in our notation.

Let $U$ be a ball containing the support of $\psi$.  By our assumptions, $C(U)=\inf_k\inf_{U}V_k(x)>0$, hence it follows that
$$
|\nabla f(x, t)|\le (2C(U))^{-1}e^{C_2(s-t)/2}\sqrt{V_k(x)}\max|\nabla\psi|.
$$
Substituting the function $u=f$ in  (\ref{ee1}) for the solution $\mu_t=\mu_t^1-\mu_t^2$,
we arrive at the following equality:
$$
\int_{\mathbb{R}^d}\psi\,d\mu_s=
\int_0^s\int_{\mathbb{R}^d}\langle b-b_k, \nabla f\rangle\,d\mu_t\,dt.
$$
Hence
$$
\int_{\mathbb{R}^d}\psi\,d\mu_s\le
\widetilde{C}\int_0^s\int_{|x|<2R}|b-b_k|\sqrt{V_k(x)}\,d|\mu_t|\,dt,
$$
where $\widetilde{C}=(2C(U))^{-1}e^{C_2T/2}\max|\nabla\psi|$ does not depend on $k$.
Letting $k\to\infty$, we conclude that
$$
\int_{\mathbb{R}^d}\psi\,d\mu_s\le 0.
$$
Recall that $\psi$ was an arbitrary function in $C^{\infty}_0(\mathbb{R}^d)$.
Then $\mu_s=\mu_s^1-\mu_s^2=0$ for almost all~$s$.
\end{proof}

We also give a result without restrictions on the growth
of $b_k$, but with some additional conditions on the solution.

\begin{theorem}\label{th2}
Let $\mu^1=\mu^1(dx)\,dt$ and $\mu^2=\mu^2(dx)\,dt$ be two solutions to the
Cauchy problem~{\rm (\ref{e1})}.
Assume that the measure $\mu=\mu_t(dx)\,dt$, where $\mu_t=\mu_t^1-\mu_t^2$,
satisfies the following condition{\rm:}

\, {\rm (i)} \,
$\displaystyle\lim\limits_{N\to\infty}N^{-1}\int_0^T\int_{N<|x|<2N}
|b(x, t)|\, |\mu_t|(dx)\,dt=0,$

\, {\rm (ii)} \, for every ball $U\subset\mathbb{R}^d$, one can find
numbers $C>0$ and $\delta>0$,
a sequence of vector fields $b_k\in C^{\infty}(U\times\mathbb{R}^1,\mathbb{R}^d)$ and
a sequence of positive functions $V_k\in C^1(U)$ {\rm(}all depending on the considered measure $\mu$
and the ball $U${\rm)} such that
$$
\inf_k\inf_{U}V_k(x)>0,
$$
$$
\langle b_k(x, t), \nabla V_k(x)\rangle \le
\Bigl(C-\delta|b_k(x, t)|^2
-2\max_{|\xi|=1}\langle\mathcal{B}_k(x, t)\xi, \xi\rangle\Bigr)V_k(x)
$$
for every $k$ and all $(x, t)\in U\times[0, T]$,
where $\mathcal{B}_k=(\partial_{x_i}b^j_k)_{i,j\le d}$,

\, {\rm (iii)} \, $\lim\limits_{k\to\infty}\|(b_k-b)\sqrt{V_k}\|_{L^1(|\mu_t|\,dt, \,U\times[0, T])}=0$.

Then $\mu=0$, i.e., $\mu^1=\mu^2$.
\end{theorem}

\begin{remark}\rm
Note that if the sequence of functions $|b_k|$  on $U\times[0, T]$
is uniformly bounded, by making $\delta$ smaller and~$C$ larger
we can restate the second estimate in (i) as
$$
\langle b_k(x, t), \nabla V_k(x)\rangle \le
\Bigl(C-2\max_{|\xi|=1}\langle\mathcal{B}_k(x, t)\xi, \xi\rangle\Bigr)V_k(x).
$$
It is worth noting that $V_k$ looks like a so-called Lyapunov function, but there is an important difference:
its sublevel sets need not be compact.
\end{remark}

As above, the proof of Theorem \ref{th2}
is based on the maximum principle and Holmgren's principle.

Let $\zeta\in C^{\infty}_0(\mathbb{R}^d)$ be
such that $0\le\zeta\le 1$ and there exists a
number $C(\zeta)>0$ such that $|\nabla\zeta(x)|^2\zeta^{-1}(x)\le C(\zeta)$ for every $x$.
Let $U$ be a ball  containing the support of $\zeta$.

We need an analog of Lemma~\ref{l01}.

\begin{lemma}\label{l1}
Suppose that $h\in C^{\infty}(\mathbb{R}^d\times\mathbb{R}^1,\mathbb{R}^d)$ and that
for some positive function $V\in C^1(\mathbb{R}^d)$ and numbers
$C>0$, $\delta>0$
one has
$$
\langle h(x, t), \nabla V(x)\rangle\le
\Bigl(C-2\sup_{|\xi|=1}\langle\mathcal{H}(x, t)\xi, \xi\rangle-
\delta|h(x, t)|^2\Bigr)V(x),
$$
for all $(x, t)\in U\times[0, T]$, where
$\mathcal{H}(x, t)=(\partial_{x_i}h^j(x, t))_{1\le i, j\le d}$.

Then, for any $s\in(0, T)$, the Cauchy problem
$$
\partial_tf+\zeta\langle h, \nabla f\rangle=0, \quad f|_{t=s}=\psi,
$$
where $\psi\in C^{\infty}_0(\mathbb{R}^d)$,
has a smooth solution $f$ such that
$$
|f(x, t)|\le \max_y|\psi(y)|, \quad
|\nabla f(x, t)|^2\le V(x)e^{M(s-t)}\max_y\frac{|\nabla\psi(y)|^2}{V(y)},\quad t\in [0,s],
$$
where $M=C+\delta^{-1}C(\zeta)$.
\end{lemma}
\begin{proof}
The reasoning is similar to that of Lemma~\ref{l01}, but we explain the necessary changes.
The existence of a smooth bounded solution with bounded derivatives is well known.
Set  $u=2^{-1}\sum_{k=1}^{d}|\partial_{x_{k}}f|^{2}$.
Differentiating the
equation $\partial_{t}f+\zeta\langle h, \nabla f\rangle=0$
with respect to $x_{k}$ and multiplying by
$\partial_{x_{k}}f$ we obtain
$$
\partial_{t}u+\zeta\langle h, \nabla u\rangle+\zeta\langle\mathcal{H}\nabla f,\nabla f\rangle+\langle\nabla\zeta,
\nabla f\rangle\langle h,\nabla f\rangle=0.
$$
Note that
$$
\langle\mathcal{H}\nabla f,\nabla f\rangle\le 2u\sup_{|\xi|=1}\langle\mathcal{H}(x, t)\xi, \xi\rangle,
\quad
\langle\nabla\zeta,\nabla f\rangle\langle h,\nabla f\rangle\le2|\nabla\zeta||h|u.
$$
Since
$$
2|\nabla\zeta||h|\le
\delta^{-1}\frac{|\nabla\zeta|^{2}}{\zeta_{M}}+
\delta\zeta|h|^{2}\le \delta^{-1}C(\zeta)+\delta\zeta|h|^2,
$$
we have
$$
\partial_tu+\zeta\langle h, \nabla u\rangle+
u\Bigl(2\zeta\sup_{|\xi|=1}\langle\mathcal{H}\xi, \xi\rangle+\delta\zeta|h|^2+\delta^{-1}C(\zeta)\Bigr)\ge 0.
$$
Set  $u=wV$. Then
$$
\partial_{t}w+\zeta\langle h, \nabla w\rangle+qw\ge0,
$$
where
$$
q=\zeta\Bigl(2\sup_{|\xi|=1}\langle\mathcal{H}\xi, \xi\rangle+\delta|h|^2
+\langle h, \nabla V\rangle V^{-1}\Bigr)+\delta^{-1}C(\zeta).
$$
By our assumptions  $q\leq C+\delta^{-1}C(\zeta)=M$.
Note that $w(x, s)\le\max_x |\nabla\psi(x)|^2/V(x)$.
Then again the maximum principle (see
\cite[Theorem 3.1.1]{StrV}) yields that $w(x, t)\le e^{M(s-t)}\max_y |\nabla\psi(y)|^2/V(y)$,
which completes the proof.
\end{proof}

\begin{proof}[Proof of Theorem \ref{th2}]
Let us fix a function $\varphi\in C^{\infty}_0(\mathbb{R}^d)$ such that
$0\le\varphi\le 1$, $\varphi(x)=1$ if $|x|<1$ and $\varphi(x)=0$ if $|x|>2$.
Assume also that $|\nabla\varphi(x)|^2\varphi(x)^{-1}\le C(\varphi)$ for every $x$.
Set $\varphi_{4N}(x)=\varphi(x/N)$.
For every ball $U_{4N}=\{x\colon |x|<4N\}$, there exist numbers $C>0$,
 $\delta>0$ and sequences $\{b_k\}$, $\{V_k\}$ such that all conditions
(ii) and (iii) are fulfilled on $U_{4N}$.
Let us fix $\psi\in C^{\infty}_0(\mathbb{R}^d)$ and $s\in(0, T)$ in the same full measure set $S$
as in the proof of~Theorem~\ref{th1}.
According to Lemma \ref{l1} there exists a smooth
solution $f$ to the Cauchy problem
$$
\partial_tf+\varphi_{2N}\langle b_k, \nabla f\rangle=0, \quad f|_{t=s}=\psi
$$
satisfying the estimate
$$
|f(x, t)|\le\max_y |\psi(y)|, \quad |\nabla f(x, t)|^2\le V_k(x)e^{M(s-t)}\max_y\frac{|\nabla\psi(y)|^2}{2V_k(y)},
$$
where $M=C+\delta^{-1}C(\varphi_{2N})$.
Certainly, $f$ depends on several parameters ($k$, $s$, $N$, $\psi$, etc.), which is suppressed in our notation.

Let $U$ be a ball containing the support of $\psi$.  By our assumptions, $C(U)=\inf_k\inf_{U}V_k(x)>0$, hence it follows that
$$
|\nabla f(x, t)|\le (2C(U))^{-1}e^{M(s-t)/2}\sqrt{V_k(x)}\max|\nabla\psi|.
$$
Substituting the function $u=f\varphi_N$ in  (\ref{ee1}) for the solution $\mu_t=\mu_t^1-\mu_t^2$
and noting that $\varphi_{2N}(x)=1$ if $\varphi_N(x)\neq 0$, we arrive at the following equality:
$$
\int_{\mathbb{R}^d}\psi\,d\mu_s=
\int_0^s\int_{\mathbb{R}^d}\Bigl[\varphi_N\langle b-b_k, \nabla f\rangle
+\langle b, \nabla\varphi_N\rangle f\Bigr]\,d\mu_t\,dt.
$$
Then we obtain
$$
\int_{\mathbb{R}^d}\psi\,d\mu_s\le
\int_0^s\int_{|x|<2N}\bigl[\widetilde{C}\varphi_N|b-b_k|\sqrt{V_k(x)}
+|b||\nabla\varphi_N|\bigr]\,d|\mu_t|\,dt,
$$
where $\widetilde{C}=(2C(U))^{-1}e^{MT/2}\max|\nabla\psi|$ does not depend on $k$.
Letting $k\to\infty$,
we conclude that
$$
\int_{\mathbb{R}^d}\psi\,d\mu_s\le
\int_0^s\int_{|x|<2N}|b||\nabla\varphi_N|\,d|\mu_t|\,dt.
$$
Since $|\nabla\varphi_N|\le N^{-1}\max|\nabla\varphi|$
and $\nabla\varphi_N|$ vanishes outside of the set $\{N<|x|<2N\}$, we have
$$
\int_{\mathbb{R}^d}\psi\,d\mu_s\le
N^{-1}\max|\nabla\varphi|\int_0^s\int_{N<|x|<2N}|b|\,d|\mu_t|\,dt.
$$
Letting $N\to\infty$, we obtain that
$$
\int_{\mathbb{R}^d}\psi\,d\mu_s\le 0,
$$
which means that $\mu_s=\mu_s^1-\mu_s^2=0$ for almost all~$s$.
\end{proof}

Thus, the main difficulty is to construct the required approximations.

\section{Corollaries and examples}

In this section we obtain some corollaries of the main result and consider some
examples. Our first corollary generalizes a result of~\cite{Cr} by omitting
restrictions on the structure of the boundary of the zero set of the field.

\begin{corollary}\label{col1}
Let $d=1$, $b\in C(\mathbb{R}^1)$ and $0\le b(x)\le C+C|x|$ for every $x\in\mathbb{R}^1$ and some number $C>0$.
Then the corresponding Cauchy problem {\rm (\ref{e1})} has at most one solution in the class of all
locally bounded {\rm(}possibly, signed{\rm)} measures $\mu$ on $\mathbb{R}^d\times[0, T]$
given by families of locally bounded measures $(\mu_t)_{t\in[0, T]}$  such that
\begin{equation}\label{p-b-c}
|\mu_t|(\partial Z)=0 \quad \hbox{\rm for almost all} \quad t\in[0, T],
\end{equation}
where $\partial Z$ is the boundary of the set $Z=\{x\colon\ b(x)=0\}$.
In particular, the latter holds for absolutely continuous $\mu_t$ provided that $b^{-1}(0)$ has Lebesgue
measure zero.
\end{corollary}
\begin{proof}
Set $Z^{0}=\{x\colon\ b(x)=0\}\setminus\partial Z$ and suppose that $Z^{0}$ is not empty.
Obviously, the measure $\nu_0=\nu|_{Z^{0}}$ is a stationary solution to the
equation $\partial_t\mu_t+{\rm div}(b\mu_t)=0$ and,
for every solution $\mu_t$, we have $\mu_t=\nu_0$ on $Z^{0}$, because the term with ${\rm div}\,(b\mu)$
vanishes in the domain~$Z^{0}$.
Replacing $\mu_t$ by $\mu_t-\nu_0$ we can trivially
prove the uniqueness of solutions $\mu_t$ such that $\mu_t(Z^{0})=0$.

We now prove that for all locally bounded measures $\mu=\mu_t\,dt$ such that $|\mu_t|(\partial Z)=0$ for almost
all $t\in[0, T]$ there exist appropriate $b_k$ and $V_k$ such that all conditions (i)--(iii)
in Theorem \ref{th1} are fulfilled.

Let us fix an interval $[-N, N]$.
In order to apply Theorem \ref{th1}, one has to find  suitable
approximations for $b$ on $[-N, N]$.
Let $\omega$ be the modulus of continuity of $b$ on $[-1-N, N+1]$.
If $\omega(\delta)=0$ for some $\delta>0$, then $b=const$ and one can take $b_k=b$ and $V_k=1$.
Let $\omega(\delta)>0$ if $\delta>0$.
Let $k\ge 1$, $\varrho_{1/k}(x)=k\varrho(kx)$, where
$\varrho\in C^{\infty}_0(\mathbb{R}^{1})$, $\varrho\ge 0$ and $\|\varrho\|_{L^1}=1$.
We use the following approximations:
$$
b_{k}=b*\varrho_{1/k}+\omega(k^{-1}), \quad
V_k=b_k^{-2}.
$$
Let us verify  conditions (ii) and (iii)
(since (i) is obviously true for such $b_k$ with $C_1=2C$).
Since $b_k(x)\le \max_{[-1-N, N+1]}b(x)=C_N$ on $[-N, N]$,
we have $\inf_k\inf_{[-N, N]}V_k\ge (4C_N)^{-2}>0$.
Moreover, $b_kV_k'=-2b_k'V_k$. Thus, condition (ii) in Theorem \ref{th1} is fulfilled.

Set $g_k(x):=|b-b_k|\sqrt{V_k}=|b-b_k|b_k^{-1}$.
Note that if $b(x)=0$, then $g_k(x)=1$. If $b(x)\neq 0$, then $g_k(x)\to 0$.
Finally,
$$
g_k=\frac{|b-b_k|}{b_k}\le
\frac{2\omega(k^{-1})}{b*\varrho_{1/k}+\omega(k^{-1})}\le 2.
$$
This yields the equality
$$
\lim_{k\to\infty}\int_0^T\int_{-N}^{N} |b-b_k|\sqrt{V_k}\,d|\mu_t|\,dt=
\int_0^T\int_{-N}^N\chi_{\{b=0\}}(x)\,|\mu_t|(dx)\,dt=0.
$$
So condition (iii) is fulfilled.
\end{proof}

\begin{remark}\label{r2}\rm
The case of nonpositive $b$  can be reduced to the considered situation: it suffices
to replace  $x$  by $-x$
and $b(x)$ by $-b(-x)$.
By the way, there is no direct generalization of the last corollary to the case of a signed drift $b$.
Indeed, it fails even for $b(x)=x^{1/3}$.
But we can assert the uniqueness for solutions concentrated on the set where $b\ge 0$ (or $b\le 0$) and satisfying
 condition (\ref{p-b-c}). This is a trivial observation because such a solution satisfies
the equation with $b^{+}=b\wedge 0$ and we can apply Corollary \ref{col1}.
There is a more interesting case where one can extract a negative part of $b$ with some better regularity.
\end{remark}

\begin{corollary}\label{col2}
Let $d=1$.
Assume that $b\in C(\mathbb{R}^1)$, $|b(x)|\le C+C|x|$ and $b=g+f$, where $f\ge 0$, $g$ is a Lipschitzian function
with Lipschitz constant $\Lambda$, and
\begin{equation}\label{pn}
g(x)<0\Longrightarrow f(x)=0.
\end{equation}
Then the corresponding Cauchy problem {\rm (\ref{e1})}
has at most one solution in the class
of all locally bounded measures $\mu$
given by families of locally bounded measures $(\mu_t)_{t\in[0, T]}$
such that
\begin{equation*}
|\mu_t|(\partial Z_f)=0 \quad \hbox{\rm for almost all} \quad t\in[0, T],
\end{equation*}
where $\partial Z_f$ is the boundary of the set $Z_f=\{x\colon\ f(x)=0\}$.
\end{corollary}
\begin{proof}
As in the previous corollary, we prove that all conditions (i)--(iii) in Theorem \ref{th1}
are fulfilled for all locally bounded measures $\mu=\mu_t\,dt$ such that $|\mu_t|(\partial Z_f)=0$ for almost
all $t\in[0, T]$.

There exists a $1$-Lipschitzian function $h$ such that $h\ge 0$, $h(x)=0$ if $f(x)>0$, and
the set of zeros of $h+f$ coincides with $\partial Z_f$.
Let $\widetilde{g}=g-h$ and $\widetilde{f}=f+h$.
As above, we fix an interval $[-N, N]$ and find  suitable approximations of $b$ on $[-N, N]$.
Let $\omega$ be the modulus of continuity of $f$ on $[-1-N, N+1]$.
Let $k\ge 1$, $\varrho_{1/k}(x)=k\varrho(kx)$, where
$\varrho\in C^{\infty}_0(\mathbb{R}^{1})$, $\varrho\ge 0$, $\varrho(x)=0$ if $|x|<1$, and $\|\varrho\|_{L^1}=1$.
Set
$$
\varepsilon_k=8\Lambda k^{-1/2},
$$
$$
\widetilde{f}_k=(f-2\omega(3/k))^{+}+h, \
b_k=(\widetilde{g}+\widetilde{f}_k)*\varrho_{1/k}+\varepsilon_k, \
V_k=(|\widetilde{g}|*\varrho_{1/k}+\widetilde{f}_k*\varrho_{1/k}+\varepsilon_k)^{-2}.
$$
Let us verify condition (ii).
We shall prove that
\begin{equation}\label{eV}
b_kV_k'\le -2(-C+b_k')V_k
\end{equation}
for some $C>0$. Since $g$ and $|g|$ are Lipschitzian functions,
we can replace  $b_k'$ in the latter inequality
by the $|\widetilde{g}|*\varrho_{1/k}'+\widetilde{f}_k*\varrho_{1/k}'$.
Note that
$$
b_kV_k'=
-2V_k\frac{(\widetilde{g}*\varrho_{1/k}+\widetilde{f}_k*\varrho_{1/k}+\varepsilon_k)
(|\widetilde{g}|*\varrho_{1/k}'+\widetilde{f}_k*\varrho_{1/k}')}
{|\widetilde{g}|*\varrho_{1/k}+\widetilde{f}_k*\varrho_{1/k}+\varepsilon_k}.
$$
In order to prove (\ref{eV}) it is enough to show that
$$
\frac{((\widetilde{g}-|\widetilde{g}|)*\varrho_{1/k})
(|\widetilde{g}|*\varrho_{1/k}'+\widetilde{f}_k*\varrho_{1/k}')}
{|\widetilde{g}|*\varrho_{1/k}+\widetilde{f}_k*\varrho_{1/k}+\varepsilon_k}\ge -\gamma
$$
for some $\gamma>0$.
If $(\widetilde{g}-|\widetilde{g}|)*\varrho_{1/k}(x)\neq 0$, then
there is a point $z$
in the interval $(x-k^{-1}, x+k^{-1})$
such that $\widetilde{g}(z)=g(z)-h(z)<0$.
If $g(z)<0$ or $h(z)>0$, then $f(z)=0$ and $(f-2\omega(3/k))^{+}(x)=0$  on the interval
$(z-3k^{-1}, z+3k^{-1})$. In particular, $\widetilde{f}_k=h$ on $(x-k^{-1}, x+k^{-1})$
and
$$
|\widetilde{f}_k*\varrho_{1/k}'(x)|=|h*\varrho_{1/k}'(x)|\le 1.
$$
We use here that $h$ is a 1-Lipschitzian function.
Moreover, for every $y\in(x-k^{-1}, x+k^{-1})$ we have
$$
|\widetilde{g}(y)|\ge |\widetilde{g}(z)|-2\Lambda k^{-1}\ge
2^{-1}(|\widetilde{g}|-\widetilde{g})(y)-4\Lambda k^{-1}.
$$
Thus, $|\widetilde{g}|*\varrho_{1/k}(x)\ge 2^{-1}(|\widetilde{g}|-\widetilde{g})*\varrho_{1/k}(x)-4\Lambda k^{-1}$ and
$$
|\widetilde{g}|*\varrho_{1/k}(x)
+\widetilde{f}_k*\varrho_{1/k}(x)+\varepsilon_k\ge 2^{-1}(|\widetilde{g}|-\widetilde{g})*\varrho_{1/k}(x),
$$
because $\varepsilon_k-4\Lambda k^{-1}>0$.
Therefore,  we have
$$
\Biggl|\frac{((\widetilde{g}-|\widetilde{g}|)*\varrho_{1/k})
(|\widetilde{g}|*\varrho_{1/k}'+\widetilde{f}_k*\varrho_{1/k}')}
{|\widetilde{g}|*\varrho_{1/k}+\widetilde{f}_k*\varrho_{1/k}+\varepsilon_k}\Biggr|\le 2(\Lambda+1).
$$
Thus, condition (ii) is fulfilled.
Let us verify condition (iii).
We have
$$
|b-b_k|\sqrt{V_k}=
\frac{|\widetilde{g}*\varrho_{1/k}-\widetilde g+\widetilde{f}_k*\varrho_{1/k}-\widetilde{f}+\varepsilon_k|}
{|\widetilde{g}|*\varrho_{1/k}+\widetilde{f}_k*\varrho_{1/k}+\varepsilon_k}.
$$
If $\widetilde{f}(x)=0$, then $\widetilde{f}_k*\varrho_{1/k}(x)=0$ and
$$
-\varepsilon_k^{-1}\Lambda k^{-1}+1\le |b(x)-b_k(x)|\sqrt{V_k(x)}\le \varepsilon_k^{-1}\Lambda k^{-1}+1.
$$
Since $\varepsilon_k^{-1}\Lambda k^{-1}\to 0$, we have $|b(x)-b_k(x)|\sqrt{V_k(x)}\to 1$.
If $\widetilde{f}(x)\neq 0$, then $|b-b_k|\sqrt{V_k}\to 0$. Thus, we obtain
$$
\lim_{k\to\infty}\int_0^T\int_{-N}^N |b-b_k|\sqrt{V_k}\,d|\mu_t|\,dt=
\int_0^T\int_{-N}^N\chi_{\{\widetilde{f}=0\}}(x)\,|\mu_t|(dx)\,dt=0
$$
and condition (iii) is fulfilled as well.
\end{proof}

\begin{remark}\rm
(i) In the same way one can prove the above result under the
assumptions that $b=g+f$, where
$$
(g(x+y)-g(x))y\le C|y|^2
$$
and $f$ is an increasing nonnegative function that  is   constant on the set $\{x\colon\ g(x)<0\}$.

(ii)
The same assertion is true in the case, where $b=g-f$,
$g$ is a Lipschitzian function and $f$ satisfies all conditions of the last corollary.

(iii)
According to Corollaries \ref{col1} and \ref{col2} the uniqueness holds in the class of solutions
with the following property: ``you must go if you can''.
\end{remark}

Let us now consider the case $d\ge 2$. First we mention a known example of uniqueness
(see, e.g.,~\cite{AB}), in which it is easy to check our hypotheses.

\begin{example}\label{col3}
\rm
Assume that $b\in C(\mathbb{R}^d\times\mathbb{R}^1, \mathbb{R}^d)$ and there exist
numbers $C_1$, $C_2$ and $C_3$ such that
$$
|b(x, t)|\le C_1+C_2|x|, \quad
\langle b(x+\xi, t)-b(x, t), \xi\rangle\le C_3|\xi|^2
$$
for all $(x, t)\in \mathbb{R}^d\times[0, T]$ and every $\xi\in\mathbb{R}^d$.
Then there exists at most one solution to (\ref{e1}) in the class
of all locally bounded measures $\mu$
given by families of locally bounded measures $(\mu_t)_{t\in[0, T]}$.
\end{example}
\begin{proof}
All hypotheses of Theorem \ref{th1} are fulfilled with the functions
$V_k=1$ and $b_k=b*\varrho_{1/k}$, where $\varrho_{1/k}=k^d\varrho(kx)$ and
$\varrho\in C^{\infty}_0(\mathbb{R}^{d+1})$, $\varrho\ge 0$, $\|\varrho\|_{L^1}=1$.
\end{proof}

Let us consider a more specific situation, where
$$
b(x)=-\beta(|x|^2)x
$$
is a radially symmetric vector field.

\begin{example}\label{col4}
\rm
Assume that $\beta\ge 0$ is bounded and continuous on $[0, +\infty)$ and that
for every interval $[0, N]$ there exists a number $\Lambda_N>0$ such that
$s\mapsto \beta(s)-\Lambda_Ns$ is a decreasing function on $[0, N]$.
Then the Cauchy problem {\rm (\ref{e1})} with $b(x)=-\beta(|x|^2)x$
has at most one solution in the class
of all locally bounded measures $\mu$ given by families of locally bounded
measures $(\mu_t)_{t\in[0, T]}$ such that
\begin{equation*}
|\mu_t|(\partial Z\setminus\{0\})=0 \quad \hbox{\rm for almost all} \quad t\in[0, T],
\end{equation*}
where $\partial Z$ is the boundary of the set $Z=\{x\colon\ \beta(|x|^2)=0\}$.
\end{example}
\begin{proof}
Let us fix a natural number $N$.
Let $\omega$ be the modulus of continuity of $\beta$ on $[0, N+1]$.
If $\omega(\delta)=0$ for some $\delta>0$, then $\beta=const$ on $[0, N]$
and one can take $b_k=b$ and $V_k=1$.
Let us consider the case where $\omega(\delta)>0$ if $\delta>0$.
Applying the
reasoning from the proof of Corollary~\ref{col1},
we arrive at the case where $\mu_t=0$ on the set $\{x\colon\ \beta(|x|^2)=0\}\setminus\partial Z$.
Let
$$
b_k(x)=-(\beta_k(|x|^2)+\omega(k^{-1}))x,
\quad
V_k(x)=(\beta_k(|x|^2)+\omega(k^{-1}))^{-2},
$$
where $\beta_k(z)=\beta*\varrho_{1/k}$.
Then
$$
\langle b_k(x),\nabla V_k(x)\rangle=4|x|^2\beta_k'(|x|^2)V_k(x)
$$
and
\begin{align*}
\sum_{i,j}\partial_{x_i}b_k^j(x)h_ih_j &=
-2\beta_k'(|x|^2)\sum_{i,j}x_ih_ix_jh_j-\bigl(\beta_k(|x|^2)+\omega(k^{-1})\bigr)|h|^2
\\
&\le-2(\beta_k'(|x|^2)\wedge 0)|x|^2|h|^2.
\end{align*}
Note also that $\beta_k'(z)\le \Lambda_{N+1}$ if $z\in[0, N]$. Hence
$4|x|^2\beta_k'(|x|^2)\le 4N\Lambda_{N+1}$ and
$$
\langle b_k(x),\nabla V_k(x)\rangle\le
\Bigl(4N\Lambda_{N+1}+4(\beta_k'(|x|^2)\wedge 0)|x|^2\Bigr)V_k(x)
$$
if $|x|\le\sqrt{N}$. Thus, condition (ii) in the definition
of $\mathcal{M}_b$ required in Theorem \ref{th1} holds.
Let us verify condition~(iii). We have
$$
g_k(x)=|b_k(x)-b(x)|\sqrt{V_k(x)}=
\frac{|\beta_k(|x|^2)-\beta(|x|^2)+\omega(k^{-1})||x|}{\beta_k(|x|^2)+\omega(k^{-1})}.
$$
Note that  if  $x\neq 0$ and $\beta(|x|^2)=0$, then $g_k(x)=1$. If $\beta(|x|^2)\neq 0$, then $g_k(x)\to 0$.
Since $g_k(x)\le 2|x|$, we obtain the equality
$$
\lim_{k\to\infty}\int_0^T\int_{|x|\le\sqrt{N}}g_k(x)\,|\mu_t|(dx)\,dt
=\int_0^T\int_{|x|\le\sqrt{N}}\chi_{\{\beta(|x|^2)=0\}}\,|\mu_t|(dx)\,dt=0.
$$
Thus, condition (iii) is fulfilled as well.
\end{proof}

\begin{remark}\rm
In the same way one can consider a more general situation:
$$b(x)=-\beta(W(x))\nabla W(x),$$
where $W$ is a smooth bounded function on $\mathbb{R}^d$ with bounded derivatives
and $\beta$ is a nonnegative continuous function on $\mathbb{R}^1$ and
such that $s\mapsto \beta(s)-\Lambda s$ is a decreasing function for some number $\Lambda>0$.
Then the uniqueness holds in the class of all
bounded measures $\mu$ given by families of bounded measures $(\mu_t)_{t\in[0, T]}$
such that
$$
|\mu_t|(\partial Z\setminus \mathcal{W})=0 \quad \hbox{\rm for almost all} \quad t\in[0, T],
$$
where $\partial Z$ is the boundary of the set $Z=\{x\colon\ \beta(W(x))=0\}$ and
$\mathcal{W}=\{x\colon\ |\nabla W(x)|=0\}$.

Indeed, we can take
$$
b_k(x)=-(\beta_k(W(x))+\omega(k^{-1}))\nabla W(x) \quad {\rm and} \quad
V_k(x)=(\beta_k(W(x))+\omega(k^{-1}))^{-2},
$$
where $\beta_k(z)=\beta*\varrho_{1/k}$, and repeat the reasoning
from the proof of the previous example.
\end{remark}

Finally, let us discuss the existence of solutions possessing our uniqueness properties.

We observe that not every approximation of $b$ yields a solution from
the considered uniqueness class. Indeed, let $b(0)=0$ and take the approximations by smooth
functions $b_k$ such that $b_k(0)=0$. Then for $\nu=\delta_0$ we obtain  the unique solution $\mu_t^k=\delta_0$
for every $k$. So, the limit measure of such $\mu_t^k$ is again~$\delta_0$.

\begin{proposition}\label{pr1}
Assume that  $b\colon\, \mathbb{R}^d\times[0, T]\to\mathbb{R}^d$ is  continuous.
Suppose that there exist a sequence of vector fields
$b_k\in C^{\infty}(\mathbb{R}^d\times\mathbb{R}^1,\mathbb{R}^d)$ and a sequence of
positive functions $V_k\in C^1(\mathbb{R}^d)$ such that

\, {\rm (i)} \,
there exist numbers $C_1$, $C_2$ and $C_3$ such that  $V_k\le C_1V_m+C_2$ for every $k\le m$ and
$$
\langle b_k(x, t), \nabla V_k(x)\rangle \le C_3V_k(x)
\quad \forall\, (x, t)\in\mathbb{R}^d\times[0, T],
$$

\, {\rm (ii)} \,
the function $(1+|x|)^{-1}|b_k(x, t)|$ is bounded for every $k$ and $b_k\to b$ uniformly on $U\times[0, T]$ for every ball $U$.

Then, for every initial nonnegative finite measure $\nu$ such that $\sup_k\|V_k\|_{L^1(\nu)}<\infty$,
there exists a family of nonnegative finite measures $(\mu_t)_{t\in[0, T]}$
solving the Cauchy problem {\rm(\ref{e1})}
with the following property{\rm:}
for every ball $U$ one has
\begin{equation}\label{ex-m}
\lim_{k\to\infty}\int_0^T\int_{U}|b-b_{k}|\sqrt{V_{k}}\,d\mu_t\,dt=0.
\end{equation}
\end{proposition}
\begin{proof}
Since $b_k$ is a smooth vector field of linear growth, there exists a nonnegative finite
solution $(\mu_t^k)_{t\in[0, T]}$ to the Cauchy problem (\ref{e1}) with $b_k$.
Moreover, $\mu_t^k(\mathbb{R}^d)\le\nu(\mathbb{R}^d)$.
Using the standard compactness arguments and the diagonal procedure (see \cite{ManSH} for details in a similar situation),
one can find
a subsequence $\{k_l\}$ such that
on every compact set in $\mathbb{R}^d$
the sequence $\{\mu_t^{k_l}\}$  converges weakly to some solution $\mu_t$ for every $t\in[0, T]$.

Let $\psi_N(x)=\psi(x/N)$. We have
$$
\int_{\mathbb{R}^d}V_{k_l}\psi_N\,d\mu_t^{k_l}=\int_{\mathbb{R}^d}V_{k_l}\psi_N\,d\nu
+\int_0^t\int_{\mathbb{R}^d}\bigl[\langle b_{k_l},\nabla\psi_N\rangle V_{k_l}
+\psi_N\langle b_{k_l}, \nabla V_{k_l}\rangle\bigr]\,d\mu_s^{k_l}\,ds.
$$
Using  conditions (i) and (ii) and letting $N\to\infty$, we arrive at the equality
$$
\int_{\mathbb{R}^d}V_{k_l}\,d\mu_t^{k_l}=\int_{\mathbb{R}^d}V_{k_l}\,d\nu
+C_3\int_0^t\int_{\mathbb{R}^d}V_{k_l}\,d\mu_s^{k_l}\,ds.
$$
Applying Gronwall's inequality we obtain that
$$
\int_{\mathbb{R}^d}V_{k_l}\,d\mu_t^{k_l}\le e^{C_3t}\int_{\mathbb{R}^d}V_{k_l}\,d\nu.
$$
Since $V_k\le C_1V_m+C_2$ for every $k, m$ such that $k\le m$, we obtain the estimate
$$
\int_{\mathbb{R}^d}V_{k}\,d\mu_t\le\lim_{l\to\infty}\int_{\mathbb{R}^d}V_{k}\,d\mu_t^{k_l}\le
C_2\nu(\mathbb{R}^d)+C_1\lim_{l\to\infty}\int_{\mathbb{R}^d}V_{k_l}\,d\mu_t^{k_l}.
$$
Hence
$$
\int_{\mathbb{R}^d}V_k\,d\mu_t\le C_2\nu(\mathbb{R}^d)+C_1e^{C_3t}\sup_k\int_{\mathbb{R}^d}V_k\,d\nu.
$$
Finally, we have
$$
\int_0^T\int_{U}|b-b_{k}|\sqrt{V_{k}}\,d\mu_t\,dt\le
\biggl(\int_0^T\int_{U}|b-b_{k}|^2\,d\mu_t\,dt\biggr)^{1/2}
\biggl(\int_0^T\int_{U}V_{k}\,d\mu_t\,dt\biggr)^{1/2},
$$
where the first factor tends to zero and the second one is uniformly bounded.
\end{proof}

\begin{example}\rm
Let us illustrate the previous proposition in the situation of Corollary \ref{col1}.
Assume that $b$ on $\mathbb{R}^1$ satisfies the following additional condition:
$$
(b(x+y)-b(x))y\ge -C|y|^2
$$
for all $x, y$ and some $C>0$.
Then all conditions of Proposition \ref{pr1} are fulfilled with the approximations
$b_k=b*\varrho_{1/k}+3\omega(1/k)$ and $V_k=b_k^{-2}$. Indeed, one has
$$
b_k'\ge -C \quad\hbox{and}\quad
b_kV_k'=-2b_k'V_k\le 2CV_k.
$$
 Let $k>m$. Note that the sequence of numbers $\omega(k^{-1})$ is decreasing and
$$
b*\varrho_{1/k}(x)+3\omega(1/k)\ge b(x)+2\omega(1/k)\ge b(x)+2\omega(1/m)\ge b*\varrho_{1/m}(x)+w(1/m).
$$
Hence we have $V_k\le 9V_m$. Assume that
$$
\int_{\mathbb{R}^1\setminus Z^{0}}\frac{1}{b^2}\,d\nu<\infty.
$$
As above we replace $\nu$ by $\nu-\nu_0$,
where $\nu_0=\nu|_{Z^{0}}$ is a stationary solution.
Then
$$
\sup_k\int_{\mathbb{R}^1}V_k\,d\nu\le
\int_{\mathbb{R}^1\setminus Z^{0}}\frac{1}{(b+2\omega(1/k))^2}\,d\nu=
\int_{\mathbb{R}^1\setminus Z^{0}}\frac{1}{b^2}\,d\nu<\infty.
$$
By Proposition \ref{pr1} there exists a solution $(\mu_t)_{t\in[0, T]}$ such that
equality (\ref{ex-m}) holds. Moreover,
$$
\int_{\mathbb{R}^1\setminus Z^{0}}\frac{1}{b^2}\,d\mu_t<\infty
$$
and $\mu_t(\partial Z)=0$ for almost all $t\in[0, T]$.
\end{example}

\begin{remark}\rm
There is another way of constructing a solution to the Cauchy problem (\ref{e1}).
Given $\varepsilon>0$, let us consider the Cauchy problem
$$
\partial_t\mu^{\varepsilon}-\varepsilon\Delta\mu^{\varepsilon}+{\rm div}(b\mu^{\varepsilon})=0,
\quad \mu^{\varepsilon}|_{t=0}=\nu.
$$
Under suitable conditions on $b$, there exists a limit point $\mu$ of the solutions $\mu^{\varepsilon}$.
To  describe the set of all limit points $\mu$  is a well-known problem.
Assume that $d=1$, $b\ge 0$, $b$ has a single zero $b(0)=0$ and $1/b$ is Lebesgue integrable
near the origin.
Let $\nu=\delta_0$.
It turns out (see \cite{KM}) that the measure $\mu$ is given by
a family of measures $\mu_t=\delta_{x_t}$, where $x_t$
is the upper extreme solution of the Cauchy problem $\dot{x}=b(x)$, $x(0)=0$.
Note that the upper extreme solution does not stay at $x=0$ and $\mu_t(\{0\})=0$ if $t>0$.
Thus, we obtain a solution from our uniqueness class.
\end{remark}

Note also that there are many results on existence of solutions to
continuity equations based on Lyapunov functions conditions (see, e.g., \cite[Corollary 3.4]{BDPR}),
but our purpose here is to ensure existence of solutions in our uniqueness classes.
In a separate paper the infinite-dimensional case will be considered; some related results
are obtained in our forthcoming paper~\cite{BDPRS}.

\vskip .1in

{\bf Acknowledgment} Our work was supported by the DFG through SFB 701 at Bielefeld University and
the RFBR projects 14-01-00237 and 14-01-90406.

\end{document}